\documentclass{amsart}
\usepackage[utf8]{inputenc}

\usepackage[english]{babel}
\usepackage{geometry}
\usepackage{xcolor}
\usepackage{amssymb}
\usepackage[all]{xy}

\makeindex

\newtheorem{thm}{Theorem}[section]

\newtheorem{lem}[thm]{Lemma}

\newtheorem{prop}[thm]{Proposition}

\theoremstyle{definition}
\newtheorem{defn}[thm]{Definition}

\newtheorem{example}[thm]{Example}

\newtheorem{rem}[thm]{Remark}

\newtheorem{defn-thm}[thm]{Definition--Theorem}  %!!!!!!!!!!!!!!!!!!!!!!!!
\newtheorem{defn-lem}[thm]{Definition--Lemma}  %!!!!!!!!!!!!!!!!!!!!!!!!
\theoremstyle{remark}

\renewcommand{\c}[0]{{\mathbb C}}

\newcommand{\n}[0]{{\mathbb N}}
\renewcommand{\r}[0]{{\mathbb R}}

\newcommand{\p}[0]{{\mathbb P}}

\newcommand{\T}[0]{{\mathbb T}}

\newcommand{\C}{\mathbb{C}}

\def\loccoh#1.#2.#3.#4.{H^{#1}_{#2}(#3,#4)}

\DeclareMathAlphabet{\mathchanc}{OT1}{pzc}%
                                {m}{it}

\newcommand{\sym}[0]{\operatorname{Sym}}

\numberwithin{equation}{section}

\begin{document}

\title[Real identifiability vs complex identifiability]{Real identifiability vs complex identifiability}

\author[E. Angelini, C.Bocci, L. Chiantini]{Elena Angelini, Cristiano Bocci, Luca Chiantini}
\address[E.Angelini,  C. Bocci, L. Chiantini]{Dipartimento di Ingegneria dell'Informazione e Scienze 
Matematiche\\ Universit\`a di Siena\\ Via Roma 56\\ 53100 Siena, Italia}
\email{elena.angelini@unisi.it, cristiano.bocci@unisi.it, luca.chiantini@unisi.it}

%\urladdr{\blue{\tt{http://}}}

\begin{abstract} 
Let $T$ be a real tensor of (real) rank $r$. $T$ is {\it identifiable} when it has a unique decomposition in terms
of rank $1$ tensors. There are cases in which the identifiability fails over the complex field, for
general tensors of rank $r$. This behavior is quite peculiar when the rank $r$ is submaximal. Often,
the failure is due to the existence of an elliptic normal curve through general points of the corresponding
Segre, Veronese or Grassmann variety. 
We prove the existence of nonempty euclidean open subsets of some variety
of tensors of rank $r$, whose  elements have several decompositions over $\c$, but only one of them is 
formed by real summands. Thus, in the open sets, tensors are not identifiable over $\c$, but are identifiable
over $\r$.
 
We also provide  examples of  non trivial euclidean open subsets
in a whole space of symmetric tensors (of degree $7$ and $8$ in three variables) and of almost unbalanced 
tensors Segre Product ($\p^2\times \p^4\times \p^9$)
whose elements have typical real rank equal to the
complex rank, and are identifiable over $\r$, but not over $\c$.
On the contrary, we provide examples of tensors of given real rank, for which
real identifiability cannot hold in non-trivial open subsets.
\end{abstract}

\maketitle

\section{Introduction}\label{sec:intr}

Recent interest has been devoted to the uniqueness
of the decomposition of a tensor in terms of rank $1$ elements
(up to rescaling and reordering). Tensors
with a unique decomposition are called {\it identifiable}.  Tensors of a given type and fixed rank $r$
are {\it generically identifiable} when  identifiability holds in a dense open subset
of the variety of tensors of rank $r$.

Several criteria are available for generic identifiability 
(see \cite{CCi06, BocciC13, BocciCOtt13, COttVan14, COttVan, 
BallBern13}) and also for the identifiability of specific tensors 
 (see \cite{Kruskal77, BallC12, BucGinenskiLand13, CMellaOtt14,
COttVan2, DomaLath13}). All the previous methods consider tensors over the complex
field. The identifiability of real tensors has been investigated e.g.
 in \cite{DomaLath15}, \cite{BernDaleoHauensteinMour},
\cite{QiComonLim}  and \cite{COttVan2}: it turns out that generic identifiability 
over $\c$ implies generic identifiability over $\r$.

We investigate here what happens to the identifiability of real tensors {\it over $\r$}, 
when the identifiability over $\c$ fails. For {\it identifiability over $\r$} we mean that a
(real) tensor $T$ may have several decompositions in terms of complex tensors of rank $1$,
but only one of these decompositions contains only rank $1$ tensors with real entries.
Thus, only one decomposition is a {\it real decomposition}.
In the notation of \cite{MichalekMoonSturmVentura}, we are interested in sufficiently
general tensors of given rank, whose variety of real decompositions $SSP(T)_\r$ is a singleton,
while the complex variety of decompositions $VSP(T)$ is not.

\begin{table}\label{tab1}
\begin{center}
\begin{tabular} {|c|c|c|c|}
space of tensors $X$ & type & rank $r$ & notes \\ \hline & & &  \\
$(\c^2)^5$ &  & $5$ & \cite[Prop. 4.1]{BocciC13} \\
$(\c^3)^6$ & symmetric & $9$ & \cite{ArbarelloCornalba81}, \cite[Rem. 6.5]{Ci01}, \cite[Prop. 2.5]{CCi06}  \\
$(\c^4)^3$ &  & $6$ & \cite[Th.1.3]{COtt12} \\
$(\c^4)^3$ & symmetric & $5$ & \cite[Lemma 4.3]{RaneVoisin}, \cite[Prop. 2.4]{COttVan} \\
$(\c^4)^4$ & symmetric & $8$ & \cite[Rem. 4.4]{Mella06}, \cite{Ball06}\\
$(\c^{10})^3$ & skew-symmetric & $5$ & \cite[Prop. 1.9]{BernVanzo}
\end{tabular} 
\end{center}
\caption{}
\end{table}

There are cases in which general tensors of rank $r$ over $\c$
have two different decompositions in terms of rank $1$ tensors.  
Some of them are resumed in Table \ref{tab1}.
In all of these cases, the failure of generic
identifiability is due to the existence of  elliptic
normal curves passing through general sets of points of 
the space $\Sigma$ of tensors of rank $1$ (Segre or Veronese varieties).
It is known after \cite[Prop 2.5]{CCi06} that elliptic normal curves
$C$ of degree $n+1$ spanning a projective space $\p^n$ ($n=2r-1$ odd) 
have the property that a general point $P\in\p^n$ has two decompositions
in terms of $r$ points of $C$, i.e. $P$ sits in two $r$-secant
$r-1$-spaces. These two decompositions determine two complex decompositions
of a general tensor $T$ of rank $r$.

Going to tensors $T$ over the reals,  in the aforementioned cases there
are three possibilities for the two complex decompositions, namely:

\begin{itemize}
\item both decompositions are real;

\item there is one real and one non-real (auto-conjugate) decomposition;

\item both decompositions are non-real (in the sense that both contain at least one non
real element);

\end{itemize}

We will prove the existence, for {\it any} elliptic normal curve $C\subset\p^n$, of
three subsets of $\p^n$ with non-empty interior (in the euclidean topology), 
in which the previous three cases
take place respectively (see Theorem \ref{elli1} and Remark \ref{postelli}). 

The result implies  that
when there are two decompositions, due to the existence of 
families of elliptic normal curves in  
the variety of tensors of rank one, then the variety $\T$ of tensors of rank $r$
 has  non-trivial, euclidean open subsets
in which any of the three situations listed above occurs.

In particular, looking at the second case, we get the existence of a subset of $\T$,
with non trivial interior, whose elements are not identifiable over $\c$,
but  since only one of the two decompositions is real, 
these tensors are identifiable over $\r$ (see Theorem \ref{ultimo}). Thus we obtain that 
real identifiability of tensors can be different from complex identifiability,
in sets with nonzero measure.

In the last section of the paper we show other examples of spaces of tensors for
which real identifiability in a non-empty euclidean open set can be proved or 
excluded.
In particular, we provide examples of  non trivial euclidean open subsets
in a whole space of symmetric tensors, whose elements have typical real rank equal to the
complex rank, and are identifiable over $\r$, but not over $\c$ (i.e. they  have several 
decompositions over $\c$, only one of which is completely real).
Examples of this type (one of which is essentially contained in \cite
[Proposition 5.6]{MichalekMoonSturmVentura}),
prove that while  identifiability for general symmetric tensors seems quite rare over $\c$
(see \cite{GaluppiMella} for ternary form), in fact it can hold over $\r$ in sets with
nonzero measure.

We hope that these remarks could be useful for applications. Also, we wonder if
real identifiability in non trivial open sets, which is determined here only in some specific case,
could be proved in higher generality.

%\begin{ack} ...
%\end{ack}

\section{Elliptic normal quartics}\label{sec:Ell4}

In the complex projective space $ \p^3=\p^3_\c $, with homogeneous coordinates 
$ [x_{0}: x_{1}: x_{2}: x_{3}] $, an elliptic normal curve is complete intersection
of two quadrics.

Conversely, every curve $C$ (of degree $4$) defined by the 
intersection of two quadric surfaces has arithmetic
genus $1$. If $C$ is irreducible, then it corresponds to the embedding of a complex torus
of (complex) dimension $1$ inside $\p^3$, i.e. to an elliptic normal curve.

The projection $C'$ of  $C$ to a plane from a general point is an irreducible
plane quartic. By the known formula which links the arithmetic and geometric
genera of plane curves, $C'$ has two nodes. This implies that a general point
of $\p^3$ sits in two secant lines to $C$. Thus a general point of $\p^3$ has
(complex) rank $2$ with respect to $C$. The exceptions are points
which lie on some tangent line to $C$.
 
When $C=Q\cap Q'$ and $Q,Q'$ are quadrics defined over the real field, then the
homogeneous ideal of $C$ is
defined by real equations. Yet, it may happen that $C$ has no real points.

We will say that $C$ is a {\it real elliptic normal quartic} when it is defined by
two real quadratic equations {\it and} has infinitely many real points.

We prove that for {\it any} irreducible real elliptic normal
quartic there are real planes that meet $C$ in $4,2,0$ real points
respectively (this fact is probably well known, but we could not
find a reference in the literature).

\begin{rem} Given a real elliptic normal quartic $C$, there are planes which meet $C$ 
in $4$ real points: enough to take $3$ real points of $C$ and the plane through them.
\end{rem}

\begin{rem} Given a real elliptic normal quartic $C$, there are planes which meet $C$ 
in $2$ real and $2$ non-real points.

Indeed take a real line $L$ which does not meet $C$. The intersections of a non-real
plane $p$, passing through $L$, with $C$ cannot be real, since all the real points
of $p$ lie in $L$. Thus $C$ contains a non-real point $P$. Therefore it contains also
the conjugate $\bar P$. If $Q$ is any real point of $C$, then $P,\bar P$ and $Q$ are
contained in a real plane $p'$. Such plane is not tangent to $C$ at $Q$, 
for a general choice of $Q$. Thus $p'$ meets $C$ in another real point $Q'$.
\end{rem}

\begin{rem} Given a real elliptic normal quartic $C$, there are planes which meet $C$ 
in $4$ non-real, pairwise conjugate points.

The claim follows from a well known principle: a real curve cannot cross a
regular tangent line. 

Consider a secant line $L$ to $C\subset\p^3$ which meets $C$ in two non-real, conjugate
points $A_1,A_2$. The projection $C'$ of $C$ from a general real point $P_0$ of $L$ 
to a real plane $\pi$ is a real
plane curve with a node $Q$ with two non-real tangent lines, i.e. a real node which is
isolated in the real plane $\pi$. If we exhibit a line $\ell$ in the plane of $C'$ which meets
$C'$ in $Q$ and two non-real  points, we are done: the real plane
spanned by $\ell$ and $P_0$ meets $C$ in $A_1,A_2$ and in two other points which
cannot be real, for their projections from $P_0$ are non-real.

So, consider the elliptic plane quartic $C'$. It has two singular points, one of them being $Q$.
We will prove that there exists a line through $Q$ meeting $C'$ in two conjugate,
non-real points.

After a change of coordinates $x,y,z$ in $\pi$, we may assume that $C'$ is singular   
at $Q=[0:1:0]$ and $Q'=[1:0:0]$, and that $Q''=[0:0:1]$ belongs to $C'$, with tangent line $x=0$.
The existence of a line passing through $ Q $ and tangent to $C'$ in a further point $ Q'' $ follows 
from the fact that the projection from $ Q $ of $C'$ is ramified at  points which can't coincide with
$ Q $, $ Q' $ by construction.
Thus the equation of $C'$ is:
\begin{equation}\label{eqform}
axz^3+z^2(cx^2+bxy+dy^2)+z(ex^2y+fxy^2)+gx^2y^2=0.
\end{equation}
We cannot have $d=0$, otherwise the line $x=0$ is a component of $C'$.
We cannot have $a=0$, otherwise also $Q''=[0:0:1]$ is a singular point, and
$C'$ is a rational curve.

Now consider affine coordinates around $Q''$, by setting $z=1$. Consider the pencil
of vertical lines $\ell_h$ with equations $x=h$, $h\in\r$. For $h=0$, $\ell_0$ intersects $C'$ at $Q''$ 
with multiplicity $2$, by construction. After substituting $x$ with the constant $h$ in (\ref{eqform}), 
we find a quadratic equation in $y$, whose discriminant is:
$$\Delta = h(-4ad-4fah-4agh^2-4cgh^3-4cfh^2-4cdh+e^2h^3+2ebh^2+b^2h).$$    
When $h$ is small enough, the term $(-4ad-4fah-4agh^2-4cgh^3-4cfh^2-4cdh+e^2h^3+2ebh^2+b^2h)$
has the sign of $-4ad\neq 0$. Hence the sign of $\Delta$ changes as $h$ passes from the negative to the
positive semi-axis. It follows that for $\epsilon>0$ sufficiently small, either the line $x=\epsilon z$
or the line $x=-\epsilon z$ (both passing through $Q$) meets $C'$ in two non-real,
conjugate points. The claim follows.
\end{rem}

Consider a point $P\in\p^3_\r$, such that there are exactly two secant lines $l(P), l'(P)$
to $C$, passing through $P$.

We will say that $P$ is of type:

\begin{itemize}
\item[s(1)] if the intersections $l(P)\cap C, l'(P)\cap C$ are four real points.

\item[s(2)] if the intersections $l(P)\cap C$ are both real points, while $l'(P)\cap C$ are 
two non-real (conjugate) points.

\item[s(3)] if the intersections $l(P)\cap C$ are non-real, conjugate points and also
$l'(P)\cap C$ are non-real, conjugate points.

\item[s(4)] if the intersections $l(P)\cap C$ are non-real points whose conjugate
points correspond to $l'(P)\cap C$.
\end{itemize}

Notice that in cases s(1) -- s(3) the two lines $l(P),l'(P)$ are real lines, while in case
s(4) the lines $l(P),l'(P)$ are non-real, and conjugate each other.

\begin{rem}\label{nocones} Assume that $P\notin C$ is a point contained in more than two
secant or tangent lines to $C$. Then the projection of $C$ from $P$ cannot be birational, 
for it would map $C$ to an elliptic irreducible plane quartic with at least $3$ singular points.
Thus $P$ is the vertex of a cone which projects $C$ to a conic, i.e. a quadric cone.
Since $C$ is the complete intersection of two quadrics, which cannot be two
cones which intersect in a line through the vertexes, since $C$ is irreducible, then $C$ sits only in
finitely many quadric cones. It follows that, with finitely many exceptions, the 
points of $\p^3\setminus C$ sit in at most two secant or tangent lines. 
\end{rem}

\begin{rem} \label{moveplane} If $\pi$ is any plane which meets $C$ in 4 
distinct points, and the four points are real (resp. non-real,  resp. $2$ real and $2$ non-real)
then the same happens for all planes in a sufficiently small euclidean
neighborhood of $\pi$ in the space of planes in $\p^3$. 

It follows by the previous remark that one can find planes which meet $C$ in $4$ real 
(resp. $4$ non-real,  resp. $2$ real and $2$ non-real) points and miss the 
vertexes of the quadric cones passing through $C$.
\end{rem}

We can collect the remarks above in the following.

\begin{prop}\label{cases-ball} For any real elliptic normal quartic $C$ 
there are points $P\in \p^3_\r$ of all the four types  {\rm s(1) -- s(4)}.

Moreover, for each type {\rm s($i$)}, $i=1,\dots,4$, there exists a non empty open ball in 
$\p^3_\r$ entirely composed of points of type {\rm s($i$)}.\end{prop}
\begin{proof}  For the existence, take general planes $\pi_1,\pi_2,\pi_3$
such that $\pi_1\cap C$ has $4$ real points, $\pi_2\cap C$ has two real and two
non-real points and $\pi_3\cap C$ has $4$ non-real points. Call $\pi_i\cap C=
\{A_i,B_i,C_i,D_i\}$. The point of intersection of the two lines $\ell_{A_1B_1}\cap 
\ell_{C_1D_1}$ is of type s(1), because, by Remark \ref{moveplane}, moving slightly 
the plane $ \pi_{1} $ such point moves away from the vertexes of the quadric cones 
containing $ C $. 
Similarly, if $A_2,B_2$ are real and $C_2,D_2$ 
are complex conjugate, then the point of intersection of the two real lines 
$\ell_{A_2B_2} \cap \ell_{C_2D_2}$ is of type s(2). If $A_3,B_3$ and $C_3,D_3$ are 
two pairs of conjugate (non-real) points, then the point $\ell_{A_3B_3}\cap 
\ell_{C_3D_3}$ has type s(3),
while the point $\ell_{A_3D_3}\cap \ell_{B_3C_3}$ (which is real because it sits in
the intersection of two conjugate lines) has type s(4).
 
Take now a point $P\in\p^3_\r$ through which there are exactly two secant lines $l(P),l'(P)$
(this means also that $P$ lies outside the tangent developable of $C$). In a small
euclidean neighbourhood $U$ of $P$, which does not intersect the
tangent developable, one can define a map $\phi:U\to(\p^3)^{\vee}$
which sends $P_\epsilon$ to the plane spanned by the two lines 
$l(P_\epsilon),l'(P_\epsilon)$. The map is clearly continuous. 

Assume now that $l(P)\cap C,l'(P)\cap C$ are $4$ real points.
Since the  plane $\phi(P)$ meets $C$ in $4$ real points, the same is true for
planes in a small neighborhood of $\phi(P)$, whose inverse image defines
a set of points around $P$ sitting in secant lines that meet $C$ in $4$ real points.
Thus the claim holds for points $P$ of type s(1).

The same argument yields the conclusion for points of type s(2).

For points $P$ of type s(3), act as before.  It turns out that there exists a small
neighborhood of $P$ formed by points $P_\epsilon$ such that 
$l(P_\epsilon),l'(P_\epsilon)$ both meet $C$ in distinct, non-real points.
Since $l(P),l'(P)$ are both real, hence they are not conjugate each other,
the same ought to be true for points $P_\mu$ in a small
neighborhood of $P$, since the lines $l(P),l'(P)$ move continuously with $P$.
Thus, there exists a neighborhood of $P$ whose points are of type s(3).

Finally, for the case s(4), just observe that if $l(P),l'(P)$ are both non-real,
then the same holds for points $P_\epsilon$ in a suitably small
neighborhood of $P$.
\end{proof}

We provide an example which shows that, \emph{near} planes tangent to $C $ at a real point, 
there are both planes that cut $C $ in complex conjugate points
 and planes that cut $C$ in distinct real points.

\begin{example}
Let us consider the two conjugate points 
$$ A = [1:i:0:0], \qquad B = [1:-i:0:0] $$
and let us denote by $ \ell_{AB} $ the line through them, which  is defined by 
$$
\left\{\begin{array}{rcl}x_{2}&=&0\\
x_{3}&=&0\end{array}\right.
$$

Now, let us consider in $ \mathbb{P}^{3}_{\mathbb{C}} $ the real non degenerate 
quadrics $ Q_{1} $ and $ Q_{2} $ given by:
\begin{gather*} Q_{1}: x_{0}^{2}+x_{1}^{2}-x_{2}^{2}-x_{3}^{2} = 0 \\
 Q_{2}: x_{0}^{2}-x_{0}x_{3}+x_{1}^{2}-x_{1}x_{3}-2x_{2}^{2}-2x_{3}^{2} = 0. 
\end{gather*}
It's not hard to see that $ Q_{1} $ and $ Q_{2} $ are one-sheeted hyperboloids 
passing through  $ A $ and $ B $.

We are interested in the real quartic elliptic curve $ C = Q_{1} \cap Q_{2}, $
which is endowed with non-singular real points (for example, $ P = [0:1:0:-1] $). 

Consider then the pencil of planes $ \mathcal{F} $ having in common 
$ \ell_{AB} $, whose generic element is defined by $ \lambda x_{2} + \mu x_{3} = 0$,
and assume $ \lambda \not= 0 $, so that we restrict to element of $ \mathcal{F} $  given by 
$$ x_{2}=kx_{3}.$$
where we take $k$ to be a real parameter.

Our aim is to classify the intersections of these elements of $ \mathcal{F} $ with $C$,
 besides $ A $ and $ B $, that is to solve the polynomial system given by
\begin{equation}\label{eq:polsys}
\left\{\begin{array}{rcl}x_{2}-kx_{3}&=&0\\
x_{0}^{2}+x_{1}^{2}-x_{2}^{2}-x_{3}^{2}&=&0\\
x_{0}^{2}-x_{0}x_{3}+x_{1}^{2}-x_{1}x_{3}-2x_{2}^{2}-2x_{3}^{2}&=&0\end{array}\right.
\end{equation}
After some manipulation, (\ref{eq:polsys}) reduces to:
 \begin{equation}\label{eq:3}
\left\{\begin{array}{rcl}x_{2}&=&kx_{3}\\
x_{3}&=&-\frac{x_{0}+x_{1}}{1+k^{2}}\\
x_{0}^{2}+x_{1}^{2}&=&\frac{(x_{0}+x_{1})^{2}}{1+k^{2}}
\end{array}\right.
\end{equation}
Now let us consider the last equation, which is equivalent to
$$ k^{2}x_{0}^{2}-2x_{0}x_{1}+k^{2}x_{1}^{2}=0. $$

It turns out that for $ -1 < k < 1 $ the corresponding planes intersect $C$ in two distinct 
real points (beside $A$ and $B$), while for $ k > 1 $ or $ k < -1 $ the corresponding 
planes intersect $C$ in two complex  conjugate points.
For the two critical values $ k = +1 $ or $ k = -1 $, the corresponding plane intersects $C$ 
in one  (double) real point (i.e the planes are tangent to $C$ at a real point).

In particular, for $ k = 1 $, we get the plane  $x_{2}=x_{3}$ which is tangent to $C$ 
at $ D= [1:1:-1:-1] $. For any small
$\epsilon>0$, the planes corresponding to  $k=1-\epsilon$ intersect $C$ in ($A,B$ and)
two distinct  real points, while the planes corresponding to  $k=1+\epsilon$ intersect
$C$ in  two complex conjugate points.
\end{example}

\begin{rem} \label{case s(2)} Notice that, if $ P $ is a real tensor of (complex) rank two 
having a decomposition $ P=T+T' $ with a real summand $T$, then also
the second tensor $T'$ is real, i.e. the decomposition is real (and the real rank is $2$).
\end{rem}

\section{Elliptic normal curves in odd-dimensional spaces}\label{sec:Ellodd}

\begin{defn}
Let $C\subset \p^{2r-1}$ be a real elliptic (irreducible) curve of degree $2r$. 
We say that a hyperplane $H$ has type $(2t,2r-2t)$ if it intersects $C$ properly in 
$2t$ real points and $2r-2t$ complex non-real points.
\end{defn}

\begin{lem}\label{hypertype}
Let $C\subset \p^{2r-1}$ be a real elliptic (irreducible) curve  of degree $2r$. Then for any 
$t=0, \dots, r$ there exists a hyperplane of type $(2t,2r-2t)$.
\end{lem}

\begin{proof}
We prove the theorem by induction on $r$.

Start with the case $r=2$, that is a real elliptic quartic curve $C$ in $\p^3$. 
By Proposition \ref{cases-ball} there exist points of type s($i$), for $i=1, \dots, 4$. 
For such points $P$ consider the span $H=\langle l(P),l'(P)\rangle$ of the two secant lines 
$l(P),l'(P)$ to $C$, passing through $P$. Hence
\begin{itemize}
\item if $P$ has type s(1) then $H$ is a hyperplane of type $(4,0)$
\item if $P$ has type s(2) then $H$ is a hyperplane of type $(2,2)$
\item if $P$ has type s(3) or s(4) then $H$ is a hyperplane of type $(0,4)$
\end{itemize}

Assume the statement is true for $r-1$. 
Given $C\subset \p^{2r-1}$, choose two real points $Q_0,Q_1\in C$. They exist since $C$ is real. 
Then consider the line $L=\langle Q_0,Q_1\rangle$, and note that $L$ is real.
Let $\pi_L:\p^{2r-1}\to \p^{2r-3}$ be the projection from $L$. The image $C'=\pi_L(C)$ is a real 
elliptic curve of degree $2r-2$. By induction there exists a hyperplane $H'\subset \p^{2r-3}$ of type 
$(2t,2r-2-2t)$ for $C'$, for all $t=0, \dots, r-1$. The inverse image $H=\pi_L^{-1}(H')$ is a hyperplane 
of type $(2t+2,2r-2-2t)$,for all $t=0, \dots, r-1$ or equivalently, of type $(2t,2r-2t)$,
for all $t=1, \dots, r$. 

To obtain the hyperplane of type $(0,2r)$ we proceed in a similar way. We first choose two complex 
conjugate points $Q_0,Q_1\in C$, and consider the line $M=\langle Q_0,Q_1\rangle$. Note that $M$ 
is still real.
Let $\pi_M:\p^{2r-1}\to \p^{2r-3}$ be the projection from $M$. The image $C'=\pi_M(C)$ is a 
real elliptic curve of degree $2r-2$. By induction there exists a hyperplane $H'\subset \p^{2r-3}$ 
of type $(0,2r-2)$ for $C'$, for all $t=0, \dots, r-1$. The inverse image $H=\pi_M^{-1}(H')$ is a 
hyperplane of type $(0,2r)$.
\end{proof}

\begin{thm}\label{elli1}
For any real elliptic normal curve $C\subset \p^{2r-1}$ of degree $2r$ there are real points 
$P\in \p^{2r-1}$ which lie in the intersection of two $(r-1)$-spaces $\Pi_1$ and $\Pi_2$ where 
$\Pi_1$  intersects $C$ in $r$ real points,
while $\Pi_2$ intersects $C$ in $r$ points, some of them non-real.
Moreover, there exists a non empty euclidean-open subset of $\p^{2r-1}_\r$ entirely composed 
of points with the same property as $P$.
\end{thm}

\begin{proof}
By Lemma \ref{hypertype} we know there exists a hyperplane $H$ of type $(2r-2,2)$. 
Let $P_1, \dots, P_{2r}$ be the  points in $C\cap H$ and suppose that $P_1, \dots, P_{2r-2}$ 
are the real points and $P_{2r-1},P_{2r}$ are the non-real ones. Then we define
\[
\Pi_1=\langle P_1, \dots, P_d\rangle \mbox { and } \Pi_2=\langle P_{r+1}, \dots, P_{2r}\rangle 
\]

The spaces $\Pi_1$ and $\Pi_2$ lie in $H$. We claim that $\dim(\Pi_1\cap \Pi_2)=0$. Indeed, if 
$\dim(\Pi_1\cap \Pi_2)=t\geq 1$, then $\langle \Pi_1,\Pi_2\rangle$ would be a space of dimension 
at most $2r-2-t$ and taking $t$ general points $R_1, \dots, R_t$ in $C$, the hyperplane 
$\langle \Pi_1,\Pi_2, R_1, \dots, R_t\rangle$ would intersect $C$ in $2r+t$ points, which is a 
contradiction since $C$ is irreducible and non-degenerate.
The point $P$ such that $\{P\}=\Pi_1\cap \Pi_2$ is the point of the statement we are looking for.

For the second part of the statement, consider the incidence variety:
\[
W=\{(H, (P_1, \dots, P_{2r}))\,: H\in (\p^{2r-1})^\vee, P_i\in H\cap C\}
\]
where $(P_1, \dots, P_{2r})$ are ordered $2r-$uples.
Call $f_1$ the projection to the first factor and call $f_2$ the map sending $(P_1, \dots, P_{2r})$ to 
$\langle P_1, \dots, P_r\rangle \cap \langle P_{r+1}, \dots, P_{2r}\rangle$.

\[
\xymatrix{ & W \ar[dl]|{f_1} \ar[dr]|{f_2} \\ (\p^{2r-1})^\vee &&
\p^{2r-1} }
\]
Let $H$ be a hyperplane of type $(2r-2,2)$ constructed above. Notice that $f_1^{-1}(H)$ consists of
$(2r)!$ points, corresponding to the permutations of the set of $2r$ distinct points
$C\cap H$. If $G\subset (\p^{2r-1})^\vee$ is a small euclidean-open subset, containing $H$, such 
that $G$ does not intersect the set of hyperplanes which are tangent to $C$, then $f_1^{-1}(G)$ 
consists of $(2r)!$ strata, each isomorphic to $G$. Fix one of these strata $G'$, corresponding to the 
choice of an ordering of the points in $C\cap H$ where the real 
points are listed first. Then consider the (rational) map $G\to \p^{2r-1}$ defined as the composition 
$\phi:=f_{2|G'}\circ f_1^{-1}$. Clearly $\phi$ is continuous and $P=\phi(H)$, thus there exists an 
euclidean-open subset $B$ of $ \p^{2r-1}$, containing $P$, whose inverse image $\phi^{-1}$ sits 
in $G$. By construction, if $B$ is sufficiently small, for all $Q\in B$, there are two $r$-secant spaces 
to $C$ passing through $Q$, one of them meeting $C$ in $r$ real points, the other meeting $C$ in 
$r-2$ real points and $2$ non-real points.
\end{proof}

\begin{rem}\label{postelli}
With similar computations, one can prove the existence of euclidean-open subsets of $\p^{2r-1}$ 
whose real points sit
in two $r$-secant spaces to $C$, the first one meeting $C$ in $r-2a$ real points and $2a$ non-real points and the second one
meeting $C$ in $r-2b$ real and $2b$ non-real points, for any choice of $a,b$ with $2a+2b\leq r$
(including the cases where $a$ or $b$ are $0$).

Notice that, however, we are not completely free to choose where the conjugates of the $2a+2b$ 
non-real points are located.

For instance, for $C$ of degree $6$ in $\p^5$, we cannot find an euclidean-open subset of points $P$
sitting in two $3$-secant planes $\Pi_1,\Pi_2$, such that $\Pi_1\cap C=\{A,B_1,B_2\}$, 
$\Pi_2\cap C=\{A',B'_1,B'_2\}$, $A,A'$ real and each $B_i$ conjugate to $B'_i$. Namely, $\Pi_1,\Pi_2$ 
cannot be real (e.g. $\Pi_1$ cannot contain $B'_1$, for no plane meets $C$ in $4$ points) and they 
are not conjugate each other (because $\Pi_1$ cannot contain $A'$), so $P$ cannot be real, 
if it is general enough.   
\end{rem}

We will need a relative version of Theorem \ref{elli1} to families of elliptic normal curves.

\begin{thm}\label{elli}
Consider an irreducible family $C_t$ of real elliptic normal curves of degree $2r$ in 
some big projective space $\p^M$. 
Call $\Gamma\subset\p^M$  the union of the curves $C_t$  and let $Y$ be the union of the
$\p^{2r-1}$'s $Y_t$ spanned by the curves $C_t$. Assume that a general real point $T\in Y$ sits 
in only two real $(r-1)$-spaces that meet $\Gamma$ in $r$ points. Then there exist
non trivial euclidean open sets  $U_1,U_2,U_3\subset Y$ such that, respectively:

\begin{itemize}
\item for all $T\in U_1$ one of the two (real) spaces  passing through $T$ 
  and $r$-secant to $\Gamma$ meets $\Gamma$ in $r$ real points and the other meets  
$\Gamma$ in some non-real point;

\item for all $T\in U_2$ both of the two (real) spaces  passing through $T$ 
  and $r$-secant to $\Gamma$ meet $\Gamma$ in some non-real point;

\item for all $T\in U_3$ both of the two (real) spaces  passing through $T$ 
 and $r$-secant  to $\Gamma$ meet $\Gamma$ in $r$ real points.
\end{itemize}
\end{thm}
\begin{proof} Fix a curve $C_0$ general in the family, which spans the space $Y_0$ and let 
$T$ be a general real point of the euclidean open subset of $Y_0$ found in Theorem \ref{elli1}. 
The two (real) $r$-secant spaces to $\Gamma$ passing through $T$ are indeed the two $r$-secant 
spaces to $C_0$. By moving $T$ continuously in a family $T_t$ of real points in $Y$, the points of 
intersection of $\Gamma$ with the two $r$-secant
spaces passing through $T_t$ also move continuously over $\r$. The claim follows from Theorem 
\ref{elli1} and Remark \ref{postelli}. 
\end{proof}

\section{Identifiability over $\r$ and elliptic curves}\label{sec:NonIdR}

In this section we apply the previous construction to the varieties listed in the table of the introduction. 
In the papers listed in the table one finds the proofs that the corresponding tensors have a  decomposition  
as follows.

\begin{rem} \label{situat}
Fix any space of tensors $X$ listed in Table \ref{tab1} and let $r$ be the corresponding value of the rank
and  $\Sigma$ the variety  of tensors of rank $1$ in $X$. Then through $r$ general points of $\Sigma$ 
there exists a unique elliptic normal curve $C$ of degree $d=2r$. A general tensor $T\in X$ 
of rank $r$ has exactly two decompositions as a sum of $r$ points in $\Sigma$. The two decompositions 
are  obtained as follows. Fix one decomposition  $P_1,\dots,P_r$ of $T$. Since $T$ is general,  
the $P_i's$ determine an elliptic curve $C\subset\Sigma$, which spans a $\p^{2r-1}$ where $T$ lies. 
There are two $r$-secant spaces to $\Sigma$ containing $T$: they determine the two decompositions
of $T$.

Notice that the elliptic curve $C$ described above is irreducible, for a general choice of $T$.
Indeed one can interchange any of the points $P_i$, by moving them in the irreducible variety $\Sigma$.
\end{rem}

\begin{prop} \label{realell}
Fix any space of tensors $X$ listed in Table \ref{tab1}. Let $r$ be the corresponding value of the rank
and let $\Sigma$ be the variety  of tensors of rank $1$ in $X$. Then for $r$ general real points of $\Sigma$ 
the unique elliptic normal curve $C$ of degree $d=2r$ passing through them is real.
\end{prop}
\begin{proof} First observe that the  existence of the unique curve $C$ fails for a choice of the $r$
points in a Zariski closed subset of $\Sigma^r$, thus it does not fail for a general choice of $r$ real points.

Now, it is clear that if the $r$ points are real, then the curve is real: otherwise the points would lie also
in the conjugate curve.

Alternatively, one can prove the claim following step by step the construction of the elliptic 
curve described in the papers listed in Table \ref{tab1}. For instance, for $(\c^2)^5$ and $r=5$, 
one obtains the elliptic curve through $P_1,\dots, P_5$ by taking  a $2$-dimensional family of elliptic 
curves in the Veronese embedding of $\p^1\times\p^1\times\p^1$ passing through $5$ general points 
(by taking hyperplane sections) and choosing the element of the family which embeds in the remaining 
two copies $\p^1\times\p^1$, passing through the projections of the $5$ points to those copies 
(see \cite[Proposition 4.1]{BocciC13}). 
Clearly when the points $P_i$'s are real, then also the curve $C$ comes out to be real.
\end{proof}

\begin{thm}\label{ultimo}  Let $X$ be any of the space of tensors listed in Table \ref{tab1} and 
consider the  corresponding value of the rank $r$.

Then there exist non trivial euclidean open subsets $U_1,U_2,U_3$ of the variety of real tensors of 
complex rank $r$ in $X$, 
whose elements $T$ have two decompositions, and:
\begin{itemize}
\item  for all $T\in U_1$, one decomposition is real and one is not
(thus the real rank is $r$ and $T$ is identifiable over $\r$);

\item for all $T\in U_2$, both decompositions are real (thus the real rank is $r$);

\item for all $T\in U_3$, both decompositions are non real (thus the real rank is bigger than $r$);
\end{itemize}
\end{thm}
\begin{proof} To prove the existence of $U_1$, fix a real tensor of real rank $r$ which is equal 
to the complex rank. 
By \cite[Lemma 5.2]{QiComonLim} there exists  an euclidean open subset $U$ in the space of tensors of 
$X$ with rank $r$, having such a property. The general real tensor in $U$ determines a set of $r$ real 
points in $\Sigma$, through which there exists a real elliptic curve $C$ which determines a linear space 
$L$ of dimension $2r-1$. $L$ sits in the closure of the variety of tensors of rank $r$ in $X$  and there 
exists a euclidean  open subset of $L$ formed by points $T$ such that only one of the two 
decompositions of $T$ with  respect to $\Sigma$ is real.
When $T$ moves in a Zariski open subset of the variety of tensors of rank $r$, the curve $C$ 
moves in an algebraic family. The claim now follows from Theorem \ref{elli}. 

The existence of $U_2$ and $U_3$ can be proved similarly.
\end{proof}

\section{Further examples}\label{furex}

In this section, we outline a general picture of cases in which the existence of euclidean open sets
whose elements are not identifiable over $\c$ but are identifiable over $\r$ can be excluded or
confirmed.

In particular, we will show that complex identifiability can be different from real identifiability
also in non trivial open subsets of tensors with generic rank.

Let us start by noticing, in a couple of remarks, that when a general tensor of given complex rank $r$ has infinitely
many decompositions with $r$ summands, then the existence of euclidean open sets of tensors
identifiable over $\r$ can be excluded.

\begin{rem}
Let $\p$ be a space of tensors whose generic element has complex rank $r$ and
infinitely many decompositions (up to permutations and rescaling) as a sum of $r$ tensors of rank $1$.
Then there are no euclidean open subsets of $\p$ whose general real elements $T$
have real rank $r$ and are identifiable over $\r$.

Namely, consider the abstract secant variety 
$$I=\{(P_1,\dots,P_r, T):\  rank(P_i)=1\  \forall i,\  T\in\langle P_1,\dots,P_r\rangle\},$$
(notice that the product is non-symmetric)
equipped with the natural projection $\pi:I\to \p$, which, by assumptions, dominates
$\p$, with positive dimensional fibers.
The singularities of the irreducible variety $I$ are contained in the locus where the points $P_i$ become
linearly dependent.  Thus, if the general real tensor $T$ of an euclidean open subset $U\subset\p$ 
has real rank $r$, then the fiber $\pi^{-1}(T)$, which has positive dimension, contains a smooth real point of $I$.
After taking a desingularization of $I$ and considering the Theorem of generic smoothness, it follows that 
$\pi^{-1}(T)$  is smooth at one real point.
Hence $\pi^{-1}(T)$ contains infinitely many real points (see \cite{BCR}, Section 2).
As a consequence, $T$ has infinitely many real decompositions.
\end{rem}

\begin{rem}
A totally analogue argument proves that when a general tensor $T$ of a given {\it sub-generic} rank $r$ has
infinitely many decompositions (up to permutations and rescaling) as a sum of $r$ tensors of rank $1$,
then there are no euclidean open subsets of the variety of tensors of rank $r$ whose general real elements
is identifiable over $\r$.

Notice that this case holds when the $r$-secant variety of a Segre or Veronese variety $X$ has
dimension smaller than the expected value, i.e. $X$ is $r$-defective. There are several known examples of
defective Segre or Veronese varieties, e.g. unbalanced Segre product (\cite {BocciCOtt13} section 8).
\end{rem}

Next, we point our attention to cases in which the general tensor of given rank has only a finite
number ($>1$) of decompositions over $\c$. 

\begin{rem}\label{rem:procedure}
A computer-aided procedure, based on \emph{homotopy continuation techniques} and \emph{monodromy loops}, implemented in the  softwares Bertini \cite{Bertini} and Matlab, in the spirit of \cite{AngeliniGaluppiMellaOtt} and \cite{BernDaleoHauensteinMour}, 
allowed us to find specific examples of symmetric tensors which are identifiable from the real point of view but not identifiable from the complex side. 

More precisely, we consider a general symmetric tensor $ T \in \p(Sym^{d}(\r^{n+1})) $ with rank $ r $ over $ \c $. We assume that $ d,n,r \in \n $ are such that $ r(n+1) = \binom{n+d}{d} $ (\emph{perfect case}) and they satisfy no exceptions of the Theorem of Alexander-Hirschowitz \cite{BrambOtt08}. 

We focus on the equation:
\begin{equation}\label{eq:form}
T - \lambda_{1}\ell_{1}^{d} - \ldots - \lambda_{r}\ell_{r}^{d} = 0
\end{equation}
with unknowns $ \ell_{i} = x_{0}+\sum_{h=1}^{n}l_{h}^{i}x_{h} \in (\p^{1}_{\c})^{\vee} $ and $ \lambda_{i} \in \c  $, for $ i \in \{1,\ldots,r\} $. By means of the identity principle of polynomials, (\ref{eq:form}) produces a (\emph{square}) polynomial system with $ \binom{n+d}{d} $ equations and unknowns, which we denote by $ \mathcal{T}_{(T)}([l_{1}^{1},\ldots,l_{n}^{1},\lambda_{1}], \ldots, [l_{1}^{r},\ldots,l_{n}^{r},\lambda_{r}]) $. Our aim is to determine the number of real solutions of $ \mathcal{T}_{(T)} $.

In practice, to get a general $ T \in \p(Sym^{d}(\r^{n+1})) $, we substitute to $ ([l_{1}^{1},\ldots,l_{n}^{1},\lambda_{1}], \ldots, [l_{1}^{r},\ldots,l_{n}^{r},\\ \lambda_{r}]) $ in $ \mathcal{T}_{(T)} $ random real numbers $ ([\overline{l}_{1}^{1},\ldots,\overline{l}_{n}^{1},\overline{\lambda}_{1}], \ldots, [\overline{l}_{1}^{r},\ldots,\overline{l}_{n}^{r},\overline{\lambda}_{r}]) $  and we compute the corresponding $ \overline{T} $, whose coefficients are called \emph{start parameters}. By construction, $ ([\overline{l}_{1}^{1},\ldots,\overline{l}_{n}^{1},\overline{\lambda}_{1}], \ldots, [\overline{l}_{1}^{r},\\ \ldots,\overline{l}_{n}^{r},\overline{\lambda}_{r}]) $ is a real solution of $ \mathcal{T}_{(\overline{T})} $, i.e. a real \emph{startpoint}. 

Therefore we consider two square polynomial systems $ \mathcal{T}_{1} $ and $ \mathcal{T}_{2} $ obtained from $ \mathcal{T}_{(\overline{T})}$ by replacing the start parameters with random complex numbers and we construct $ 3 $ segment homotopies
$$ H_{i} : \C^{r(n+1)} \times [0,1] \to \C^{r(n+1)}, \, i \in \{0,1,2\} $$
in a way such that $H_{0}$ is between $ \mathcal{T}_{(\overline{T})} $ and $ \mathcal{T}_{1} $, $ H_{1} $ is between $ \mathcal{T}_{1} $ and $ \mathcal{T}_{2} $, $ H_{2} $ is between $ \overline{T}_{2} $ and $ \mathcal{T}_{(\overline{T})} $. With $ H_{0} $, we connect the startpoint $ ([\overline{l}_{1}^{1},\ldots,\overline{l}_{n}^{1},\overline{\lambda}_{1}], \ldots, [\overline{l}_{1}^{r}, \ldots,\overline{l}_{n}^{r},\overline{\lambda}_{r}]) $ to a solution (\emph{endpoint}) of $ \mathcal{T}_{1} $, which therefore becomes a startpoint for the second step given by $ H_{1} $, and so on. 

At the end of this triangle-loop, we check if the output is a solution of the polynomial system $ \mathcal{T}_{(\overline{T})} $ different from the starting one. If this is not the case, we restart the procedure: indeed, as stated before this remark, we are assuming that the number of solutions of $ \mathcal{T}_{(\overline{T})} $ is finite but greater than $ 1 $. Otherwise we iterate this technique with these two startingpoints, and so on. 

At certain point, the number of solutions of $ \mathcal{T}_{(\overline{T})} $ stabilizes and, as output, we get the set of all the decompositions of $ \overline{T} $ over $ \c $. If $ ([\overline{l}_{1}^{1},\ldots,\overline{l}_{n}^{1},\overline{\lambda}_{1}], \ldots, [\overline{l}_{1}^{r}, \ldots,\overline{l}_{n}^{r},\overline{\lambda}_{r}]) $ is the unique solution with real entries, then $ \overline{T} $ is a real symmetric tensor which is identifiable over $ \r $ but not over $ \c $.    
\end{rem}

In the following, we describe two examples in which we apply the procedure of Remark \ref{rem:procedure}, obtaining that real identifiability holds in euclidean open subsets composed by tensors of generic rank. 

\begin{example} Let $T$ be a general symmetric tensor of type $3\times 3\times\dots \times 3$ ($7$ times). Dixon and Stuart proved that $T$ has rank $12$ over $\c$ and it can be decomposed in exactly $5$ ways as a sum of $12$ tensors of 
rank $1$ (see \cite[Theorem 3.1]{RaneSchr00}). 

By applying our procedure to the tensor $ \overline{T} $, arising from \ref{eq:form} with $ d = 7, \,n = 2,\, r = 12  $ and startpoint
$$  [l_{1}^{1},l_{2}^{1},\lambda_{1}] = [-3.831393646843184, 1.346964775131610\cdot10^{-1},2.425782032500251\cdot10^{2}] $$
$$  [l_{1}^{2},l_{2}^{2},\lambda_{2}] = [9.931270838081495\cdot10^{-1},-6.769701755660390\cdot10^{-1},4.146536442894879\cdot10^{2}] $$
$$  [l_{1}^{3},l_{2}^{3},\lambda_{3}] = [3.183385725212400,-7.633860595893790\cdot10^{-1},4.843082801697150\cdot10^{2}] $$
$$  [l_{1}^{4},l_{2}^{4},\lambda_{4}] = [-8.878812851871381\cdot10^{-1},9.326430222177290\cdot10^{-1},
-3.093559475729942\cdot10^{1}] $$
$$  [l_{1}^{5},l_{2}^{5},\lambda_{5}] = [-8.333546205381460\cdot10^{-1},4.787791245905811,5.913320307260028\cdot10^{2}] $$
$$  [l_{1}^{6},l_{2}^{6},\lambda_{6}] = [1.150535726607133,-7.356530267574411,1.863359371761127\cdot10^{2}] $$
$$  [l_{1}^{7},l_{2}^{7},\lambda_{7}] = [-6.333358363820080\cdot10^{-1},3.556043275765582,-6.986594239306317\cdot10^{2}] $$
$$  [l_{1}^{8},l_{2}^{8},\lambda_{8}] = [2.649721933021775,-2.942789804855117,9.082119499105495\cdot10^{1}] $$
$$  [l_{1}^{9},l_{2}^{9},\lambda_{9}] = [9.281823004496396\cdot10^{-1},5.416247221839678\cdot10^{-1},-3.774941091391256\cdot10^{1}] $$
$$  [l_{1}^{10},l_{2}^{10},\lambda_{10}] = [-3.760716164753004,1.290194389580469,-8.149598050955672\cdot10^{-1}] $$
$$  [l_{1}^{11},l_{2}^{11},\lambda_{11}] = [2.159937720250393,-1.622029661864421,5.360726064748198] $$
$$  [l_{1}^{12},l_{2}^{12},\lambda_{12}] = [-8.097853608809100\cdot10^{-1},5.078077230490563\cdot10^{-1},-1.967556570270287\cdot10^{1}] $$
we explicitly get the five solutions $ \overline{T}=T_{1},T_{2}, T_{3}, T_{4}, T_{5} $ of $ \mathcal{T}_{(\overline{T})} $ (for more details, see the ancillary file, available online: 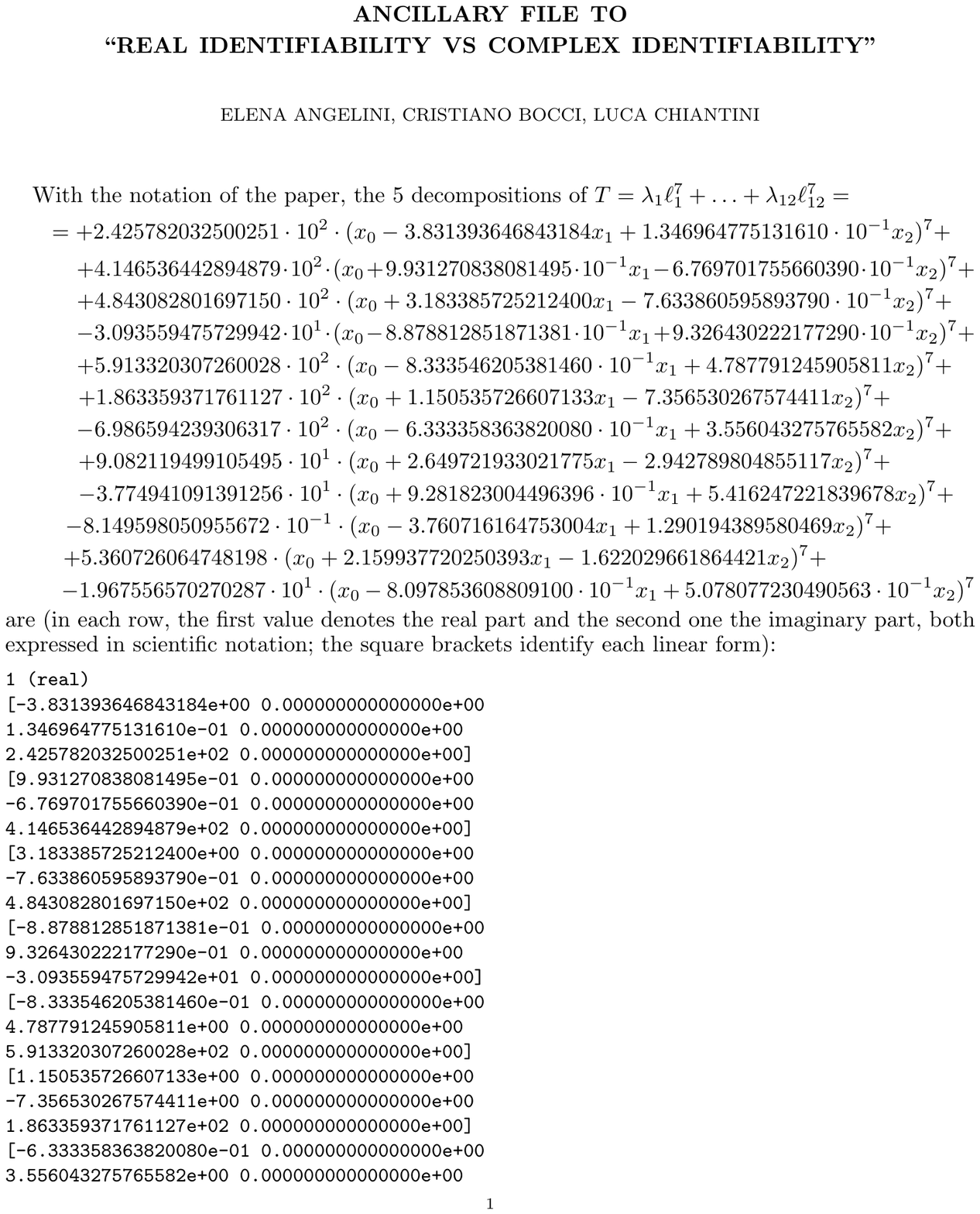): except $ T_{1} $, they are all non-real, in particular $ T_{2} $ and $ T_{3} $ are autoconjugate and, up to reordering rank-1 tensors, $T_{4}$ and $T_{5}$ are conjugate.       

Moving $\overline{T}$ in a small euclidean open subset over the reals, only one decomposition remains
real, because the property of being real is open in the set of decompositions (see 
\cite{MichalekMoonSturmVentura}, Proposition 5.6). It turns out that there is a nontrivial euclidean open subset of $\p(\sym^7(\r^3))$ whose elements have only one real decomposition, plus $4$ non-real ones.

Another instance of the same phenomenon is outlined in \cite[Example 5.6]{MichalekMoonSturmVentura}. 
The authors find one real tensor $\overline{T}\in\p(\sym^7(\r^3))$, which sits in a space $L$ of dimension $11$ 
containing one tangent line and $10$ points of the Veronese variety $v_7(\p^2)$. The tensor $\overline{T}$ has only $4$ proper decompositions, none of which is real. The fifth decomposition is only a {\it cactus} decomposition, since it contains a double point $P$. Yet, all the points of this cactus decomposition are real. 
Moving $L$ properly, $P$ splits in two real points of $v_7(\p^2)$. Thus  suitable euclidean open subsets near $\overline{T}$ are formed by real tensors with five proper decompositions, only one of which is real.
\end{example}

\begin{example} Let $T$ be a general symmetric tensor of type $ 3 \times 3 \times \cdots \times 3 $ ($ 8 $ times). 
Then it is well known that $T$ has rank $15$ over $ \c $ and it can be decomposed in exactly $16$ ways as a sum of $15$ rank $1$ tensors  (see \cite[Theorem 1.7]{RaneSchr00}).

Arguing as above, but with much more computational effort, one finds an euclidean open subset of $\p(\sym^8(\r^3))$ whose elements have only one real decomposition, plus $15$ non-real ones
(for more details, see the ancillary file, available online: arxiv.org/src/1608.07197v3/anc/CRidentAnc.pdf).
\end{example}

We show, on the other hand, that there are also cases in which no euclidean open subset of the variety of
tensors with fixed rank can be filled by tensors with only one real decomposition.

\begin{example}\label{norealid} Consider the variety of tensors $T$ of rank $5$ (submaximal) in
$\c^3\otimes\c^3\otimes\c^6$. It is known that the general tensor of this type
is not identifiable over $\c$ (see e.g. \cite[Proposition 8.3]{BocciCOtt13}).
Namely, the decompositions of general tensor of rank $5$ in $\c^3\otimes\c^3\otimes\c^6$
correspond, under a contraction map over the last factor, to $5$ points of the Segre
variety $\p^2\times\p^2\subset \p^8$ which span a space of (projective) dimension $4$.
Since the degree of $\p^2\times\p^2$ in $\p^8$ is even ($=6$), then a linear space
spanned by $5$ general real points meets $\p^2\times\p^2$ in $6$ real points.
Thus any sufficiently general tensor $T$ of rank $5$ has $6$ different real
decompositions. 
\end{example}    
 
The situation outlined in Example \ref{norealid} indeed occurs for tensors of rank $a_q$
 in any {\it almost} unbalanced
Segre product $\p^{a_1}\times\dots\times\p^{a_q}$, with $a_q=\Pi_{i=1}^{q-1}(a_i+1)- 
\sum_{i=1}^{q-1}a_i$, whenever the degree $D$ of the Segre product
$\p^{a_1}\times\dots\times\p^{a_{q-1}}$ satisfies the condition that $D-a_q$ is odd
(see \cite[Proposition 8.3]{BocciCOtt13} again).

We wonder if for almost unbalanced products with $D-a_q$ even, one can always find an
euclidean open subset of the variety of tensors of rank $a_q$ whose elements are
identifiable over $\r$.\smallskip
 
\begin{example} At least in specific examples, we can prove the real identifiability in euclidean open
sets of real secant varieties of almost unbalanced tensors.

Take for instance tensors $T$ of type $3\times 5\times 10$ and rank $9$. The degree $D$ of the Segre embedding
of $\p^2\times\p^4$ in $\p^{14}$ is $15$. Under a contraction over the last factor, 
the decompositions of $T$  are determined by the intersections of
the Segre emebedding $X$ of   $\p^2\times\p^4$ with a linear space $L$ of dimension $8$ in $\p^{14}$.
By means of a computer-aided procedure (see the ancillary file, available online: arxiv.org/src/1608.07197v3/anc/CRidentAnc.pdf),
one can find a real linear space $L$ so that $X\cap L$ has $9$ real points and $6$ non-real points.
Consequently, one can find $9$ tensors of rank $1$ and type $3\times 5\times 10$ whose span in $\p^{149}$
contains general elements with $\binom{15}9$ decompositions, only one of which is real.
The property of intersecting $X$ in $9$ real and $6$ non-real points is maintained by moving
$L$ generically in the variety of real $8$-spaces in $\p^{14}$. Thus we find an euclidean open subset 
of the variety of tensors of type $3\times 5\times 10$  and rank $9$, whose elements are identifiable over $\r$.
\end{example}

\bibliographystyle{amsplain}
\bibliography{biblioLuca}

\end{document}